\documentclass[a4paper,oneside]{article}
\usepackage{graphicx}
\usepackage{amssymb}
\usepackage{amsmath}
\usepackage{amsthm}
\usepackage{enumerate}
\usepackage[english]{babel}
\usepackage{url}

\usepackage{flafter}

\newcommand{\Inv}[1]{\frac{1}{#1}}

\newcommand{\kgs}{\lrp{\frac{1}{K^*} + \frac{1}{G^*}}}
\newcommand{\gks}{\lrp{\frac{1}{K^*} - \frac{1}{G^*}}}

\newcommand{\kg}{\lrp{\frac{1}{K} + \frac{1}{G}}}
\newcommand{\gk}{\lrp{\frac{1}{K} - \frac{1}{G}}}

\newcommand{\di}{\dashint}
\newcommand{\dint}{\di}
\def\Xint#1{\mathchoice
{\XXint\displaystyle\textstyle{#1}}%
{\XXint\textstyle\scriptstyle{#1}}%
{\XXint\scriptstyle\scriptscriptstyle{#1}}%
{\XXint\scriptscriptstyle\scriptscriptstyle{#1}}%
\!\int}
\def\XXint#1#2#3{{\setbox0=\hbox{$#1{#2#3}{\int}$ }
\vcenter{\hbox{$#2#3$ }}\kern-.6\wd0}}

\def\dashint{\Xint-}

\newcommand{\eii}{\e_{11}}
\newcommand{\ejj}{\e_{22}}
\newcommand{\eij}{\e_{12}}

\newcommand{\eIJ}{\e_{ij}}

\newcommand{\sii}{\sigma_{11}}
\newcommand{\sjj}{\sigma_{22}}
\newcommand{\sij}{\sigma_{12}}

\newcommand{\xii}{\xi_{11}}
\newcommand{\xjj}{\xi_{22}}
\newcommand{\xij}{\xi_{12}}

\newcommand{\wij}{\omega_{12}}
\newcommand{\wji}{\omega_{21}}

\newcommand{\wIJ}{\omega_{ij}}

\newcommand{\as}[1]{\section{#1}}

\newtheorem{tm}{Theorem}
\newtheorem{lm}{Lemma}

\newtheorem{rem}{Remark}

\newcommand{\inv}{^{-1}}

\newcommand{\mop}[1]{\mathop{\mathrm{#1}}}

\newcommand{\ddt}[2]{\frac{\partial #1}{\partial #2}}

\newcommand{\set}[1]{\left\{ #1 \right\}}

\newcommand{\lrp}[1]{\left(#1\right)}

\newcommand{\lrpb}[1]{\Bigl(#1\Bigr)}
\newcommand{\ave}[1]{\langle#1\rangle}

\newcommand{\lran}[1]{\langle#1\rangle}

\newcommand{\re}[1]{(\ref{#1})}
\newcommand{\e}{{\varepsilon}}

\newcommand{\mat}[1]{\begin{pmatrix}#1\end{pmatrix}}

\newcommand{\rr}{\mathbf{R}}

\newcommand{\zz}{\mathbf{Z}}

\newcommand{\dx}{\,dx}

\newcommand{\dy}{\,dy}

\makeatletter

\newcommand{\Rmnum}[1]{\expandafter\@slowromancap\romannumeral #1@}
\makeatother

\newdimen\tildemin \tildemin=1.04em % this is the adjustable size up to which no shrinking of the ~m is performed
\providecommand{\ave}[1]{\buildrel
  \hbox{\mathsurround=0pt\relax $\smash{%  kill any math-surround
   \lower1.2ex\hbox{$% this is the adjustable measurement by which the ~m is to lowered
    \setbox 0=\hbox{$\hphantom{#1}$}%   set #1 in box 0
     \ifdim\wd 0<\tildemin   % \wd 0  is the width of box 0
       \widetilde{\hphantom{#1}}^{m}%  use the full width of #1
     \else
       \setbox 2=\hbox{$^{iii}$}%  ... by subtracting the width of the characters of your choosing
       \dimen 0=\wd 0\advance\dimen 0-\wd2\relax
       \widetilde{\hbox to\dimen 0{\hss$\hphantom{#1}$}}^{m}%
     \fi
    $}}$}\over{#1}}
\renewcommand{\ave}[1]{\buildrel
  \hbox{\mathsurround=0pt\relax $\smash{%  kill any math-surround
   \lower1.2ex\hbox{$% this is the adjustable measurement by which the ~m is to lowered
    \setbox 0=\hbox{$\hphantom{#1}$}%   set #1 in box 0
     \ifdim\wd 0<\tildemin   % \wd 0  is the width of box 0
       \widetilde{\hphantom{#1}}^{m}%  use the full width of #1
     \else
       \setbox 2=\hbox{$^{iii}$}%  ... by subtracting the width of the characters of your choosing
       \dimen 0=\wd 0\advance\dimen 0-\wd2\relax
       \widetilde{\hbox to\dimen 0{\hss$\hphantom{#1}$}}^{m}%
     \fi
    $}}$}\over{#1}}
\author{Dag Lukkassen\footnote{Narvik University College, P.O. Box 385, N-8505 Narvik, Norway, and Norut Narvik, P.O. Box 250, N-8504 Narvik, Norway.}, Annette Meidell\footnotemark[\value{footnote}], and  Klas Pettersson\footnote{Narvik University College, P.O. Box 385, N-8505 Narvik, Norway.}}

\title{An elementary proof of the Vigdergauz equations for a class of square symmetric structures}
\begin{document}
\maketitle

\begin{abstract}
For a periodically perforated structure, for which homogenization takes place
  in the linear theory of elasticity, the components of the effective elasticity
  tensor depend in general on the geometry of the holes as well as on the local
  elastic properties.
These dependencies were shown by Vigdergauz in~\cite{vigder} to be separated
  in an elementary way for one particular class of structures.
The original proof of this relation made use of the lattice approach
  to describe periodic functions using complex variables. 
In this paper we present a proof of the so-called Vigdergauz
  equations for a related class of square symmetric structures.
Our proof relies solely on the fundamental theorem of real variable calculus.
The differences between the two mentioned classes of structures are
  nontrivial which makes our result a partial generalization as well.
\end{abstract}

%{\emph{Keywords:} C. Elastic properties, C. Multiscale modeling.}

\as{Introduction}

In~\cite{vigder} Vigdergauz showed an explicit relation between the local and
  the effective elastic properties for a special class of periodic structures
  in the linear theory of elasticity.
The result relates the components of the effective elasticity tensor to the
  local elastic properties as elementary functions involving one geometric
  constant for each distinct component of the effective tensor.
For a subclass of the square symmetric planar structures with local
  elastic moduli $(K,G)$, the relation to the corresponding effective properties
  $(K^*,G^*,G^*_{45})$ can be written
  \begin{align*}
  \Inv{K^*} & = \Inv{K} + A_1\kg,\\
  \Inv{G^*} & = \Inv{G} + A_2\kg,\\
  \Inv{G^*_{45}} & = \Inv{G} + A_3\kg,
  \end{align*}
  where $(A_1,A_2,A_3)$ are geometric constants that depend solely on the
  domain of a periodicity cell of the structure.

The main restrictions defining the class of structures studies
  in~\cite{vigder} were that the periodicity cell should contain only one
  hole which is centered strictly inside the cell.
In addition the boundary of the cell was supposed to be smooth.
An example of such a structure is shown in Figure~\ref{fig:structures}(a).

In this paper we will consider the class obtained by relaxing the assumptions
  on the number of holes and the smoothness of the boundary.
More precisely, we consider the structures with positive definite square
  symmetric effective tensor, which have a cell with Lipschitz continuous
  boundary.
Moreover, we suppose that the cell can be transversed by some lines parallel
  to the coordinate axes.
The last requirement enables us to give a particularly simple proof of the
  explicit relations, which we call the Vigdergauz equations.

An example of a structure of the kind we will consider is shown in
  Figure~\ref{fig:structures}(b).
This is not of the type studied in~\cite{vigder} both because of the lack of
  smoothness and due to the number of holes in any peridicity cell.
Conversely, the structure shown in Figure~\ref{fig:structures}(a) is not in
  the class we will consider because no periodicity cell can be transversed
  by line segments.
Hence the classes have nontrivial differences and so our result gives a
  partial generalization of the Vigdergauz equations.

As an illustration we consider now the first two of the above equations.
To emphasize on the idea of proof we will make formal calculations and
  leave the precise definitions and arguments to the later sections.

Let $C$ be the domain of one period of a structure in $\rr^2$ and $Y$ the corresponding cell,
  which we assume is a rectangle.
We consider the case of linear elasticity and use the planar bulk modulus $K$ and
  shear modulus $G$ as the local material parameters. The stress $\sigma$ and the strain $\e$
  are then linearly related via the local version of the Hooke law:
  \begin{equation}
  \mat{\varepsilon _{11} \\ 
  \varepsilon _{22} \\ 
  \gamma _{12}}
  =
  \mat{
  \frac{1}{4}\left( \frac{1}{K}+\frac{1}{G}\right) & \frac{1}{4}\left( \frac{1%
  }{K}-\frac{1}{G}\right) & 0 \\ 
  \frac{1}{4}\left( \frac{1}{K}-\frac{1}{G}\right) & \frac{1}{4}\left( \frac{1%
  }{K}+\frac{1}{G}\right) & 0 \\ 
  0 & 0 & \frac{1}{G}%
  }
  \mat{\sigma _{11} \\ 
  \sigma _{22} \\ 
  \sigma _{12}},
  \end{equation}
  here written in Voigt notation and $\gamma_{12}$ is the engineering shear strain.
The matrix on the right hand side is the inverse of the representation of
  the elasticity tensor.
We write the two first equations explicitly as
  \begin{align*}
  \varepsilon _{11} &= \frac{1}{4}\left( \frac{1}{K}+\frac{1}{G}\right) \sigma
  _{11}+\frac{1}{4}\left( \frac{1}{K}-\frac{1}{G}\right) \sigma _{22}, \\
  \varepsilon _{22} &= \frac{1}{4}\left( \frac{1}{K}-\frac{1}{G}\right) \sigma
  _{11}+\frac{1}{4}\left( \frac{1}{K}+\frac{1}{G}\right) \sigma _{22}.
  \end{align*}
We assume that the strain components $\varepsilon _{ij}$, hence also the
  stress components $\sigma _{ij}$, are $Y$--periodic, where $Y=\left[ 0,l_{1}\right] \times \left[ 0,l_{2}\right]$, for some positive real
  numbers $l_1$ and $l_2$.
We also assume global quadratic symmetry.
Hence, we have the corresponding average relations
  \begin{align*}
  \left\langle \varepsilon _{11}\right\rangle &=\frac{1}{4}\left( \frac{1}{%
  K^{\ast }}+\frac{1}{G^{\ast }}\right) \left\langle \sigma
  _{11}\right\rangle +\frac{1}{4}\left( \frac{1}{K^{\ast }}-\frac{1}{%
  G^{\ast }}\right) \left\langle \sigma _{22}\right\rangle , \\
  \left\langle \varepsilon _{22}\right\rangle &=\frac{1}{4}\left( \frac{1}{%
  K^{\ast }}-\frac{1}{G^{\ast }}\right) \left\langle \sigma
  _{11}\right\rangle +\frac{1}{4}\left( \frac{1}{K^{\ast }}+\frac{1}{%
  G^{\ast }}\right) \left\langle \sigma _{22}\right\rangle,
  \end{align*}
where $\lran{\cdot}$ represents an average over $C$ or $Y$.
Here $K^*$ denotes the effective bulk modulus, and $G^*$ one of the
  two effective (transverse) shear moduli.

We assume that the possible holes in $C$ are such that there exists both horizontal 
  and vertical line segments contained in $C$, joining opposite faces of $Y$.
Then we can integrate along a horizontal line in the $Y$--cell, and find the
  value $\int \sigma _{22}\dx$.
It is possible to show that this integral is
  independent of $y.$ Similarly, we find, by integrating vertically, that
  $\int \sigma _{11}\dy$ is independent of $x$.
Thus, we find that 
  \begin{align*}
  \lran{ \sjj } & = \frac{1}{l_1}\int \sjj \dx, &
  \lran{ \sii } & = \frac{1}{l_2}\int \sii \dy.
  \end{align*}
Due to the periodicity of the deformed structure,
  $\int \eii \dx = u_{1}(l_{1},y)-u_{1}(0,y)$ and
  $\int \ejj \dy = u_{2}(x,l_{2})-u_{2}(x,0)$, are 
  independent of $x$ and $y$.
Hence, 
\begin{align}\label{assasas}
  \lran{\ejj} & = \frac{1}{l_2} \int \ejj \dy, & 
  \lran{\eii} & = \frac{1}{l_1} \int \eii \dx.
\end{align}
Let us now assume that the average vertical forces is zero, that is
  $\left\langle \sigma _{22}\right\rangle =0$.
Then, integrating along lines without holes we obtain 
  \begin{align*}
  \left\langle \varepsilon _{11}\right\rangle &=\frac{1}{l_{1}}\int
  \varepsilon _{11}\dx=\frac{1}{4}\left( \frac{1}{K}+\frac{1}{G}\right) \frac{1%
  }{l_{1}}\int \sigma _{11}\dx \\
  \left\langle \varepsilon _{22}\right\rangle &=\frac{1}{l_{2}}\int
  \varepsilon _{22}\dy=\frac{1}{4}\left( \frac{1}{K}-\frac{1}{G}\right) \frac{1%
  }{l_{2}}\int \sigma _{11}\dy+\frac{1}{4}\left( \frac{1}{K}+\frac{1}{G}\right) 
  \frac{1}{l_{2}}\int \sigma _{22}\dy.
  \end{align*}
Let $a$ and $b$ be defined by
  \begin{align*}\label{}
  (1+a)\lran{\sii} & = \frac{1}{l_1} \int \sii \dx, &
  b \lran{\sii}    & = \frac{1}{l_2} \int \sjj \dy.
  \end{align*}
Then, by the equations in \re{assasas} we obtain that
  \begin{align*}
  \text{ }\left\langle \varepsilon _{11}\right\rangle &=\frac{1}{4}\left( 
  \frac{1}{K}+\frac{1}{G}\right) (1+a)\left\langle \sigma _{11}\right\rangle, \\
  \left\langle \varepsilon _{22}\right\rangle &=\frac{1}{4}\left( \frac{1}{K}-%
  \frac{1}{G}\right) \left\langle \sigma _{11}\right\rangle +\frac{1}{4}\left( 
  \frac{1}{K}+\frac{1}{G}\right) b\left\langle \sigma _{11}\right\rangle.
  \end{align*}
Combined with the reduced average relations,
  \begin{align*}
  \lran{\eii} & = \frac{1}{4}\left( \frac{1}{%
  K^{\ast }}+\frac{1}{G^{\ast }}\right) \left\langle \sigma
  _{11}\right\rangle, &
  \left\langle \varepsilon _{22}\right\rangle & = \frac{1}{4}\left( \frac{1}{%
  K^{\ast }}-\frac{1}{G^{\ast }}\right) \left\langle \sigma
  _{11}\right\rangle,
  \end{align*}
we find that
  \begin{align*}
  \left( \frac{1}{K}+\frac{1}{G}\right) (1+a) &=\frac{1}{K^{\ast }}+\frac{1}{%
  G^{\ast }}, \\
  \left( \frac{1}{K}-\frac{1}{G}\right) +\left( \frac{1}{K}+\frac{1}{G}\right)
  b &=\frac{1}{K^{\ast }}-\frac{1}{G^{\ast }}.
  \end{align*}
Adding and subtacting these two equations yield, respectively,
  \begin{align*}
  \frac{1}{K}+\left( \frac{1}{K}+\frac{1}{G}\right) \frac{a+b}{2} &=\frac{1%
  }{K^{\ast }}, \\
  \frac{1}{G}+\left( \frac{1}{K}+\frac{1}{G}\right) \frac{a-b}{2} &=\frac{1%
  }{G^{\ast }}.
  \end{align*}
In summary, we have found the following representations of the Vigdergauz constants:
\begin{align*}
A_1 & = \frac{\frac{1}{l_1}\int \sigma_{11} \dx + \frac{1}{l_2}\int \sigma_{22} \dy}{2\lran{\sigma_{11}}} - \frac{1}{2}, \\
A_2 & = \frac{\frac{1}{l_1}\int \sigma_{11} \dx - \frac{1}{l_2}\int \sigma_{22} \dy}{2\lran{\sigma_{11}}} - \frac{1}{2}.
\end{align*}

The original proof of the above mentioned relations made use of some properties
  of the Weierstrass zeta-function and the Kolosov-Muskhelishvili potentials.
In this paper we present a proof of the Vigdergauz
  equations for a class of square symmetric structures.
Our proof relies solely on the Newton-Leibniz and the Green formulas,
  and a lemma of Michell type.
The Michell lemma can be seen as a consequence of the theorem of Stokes in the current
  setting.

The rest of this paper is organized as follows.
In the next section we give the precise definition of the class of structures
  that will be studied and set the notation.
In Section~\ref{sec:averagestresses} and~\ref{sec:quasiperiods}, we give
  some representations of the average stress and the quasiperiods for the cell
  problem in the definition of the effective tensor.
These results are used in the proof of the Vigdergauz equations in
  Section~\ref{sec:vigdergauz}.

\begin{figure}[htbp]
\begin{center}
\begin{minipage}[c]{.45\textwidth}
\centering
\includegraphics[trim = 75mm 120mm 75mm 120mm, clip, width=.9\textwidth]{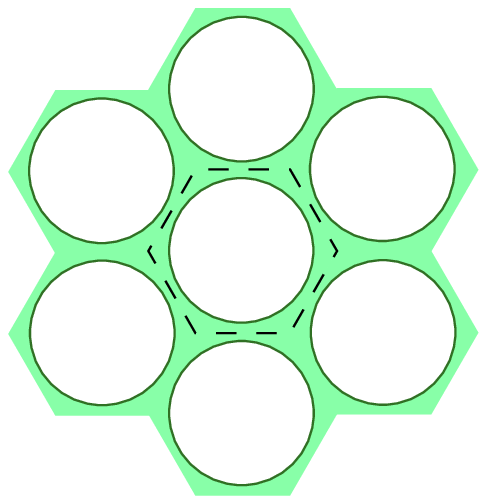}
\end{minipage}
\begin{minipage}[c]{.05\textwidth}
$\,$
\end{minipage}
\begin{minipage}[c]{.45\textwidth}
\centering
\includegraphics[trim = 73mm 120mm 73mm 120mm, clip, width=.9\textwidth]{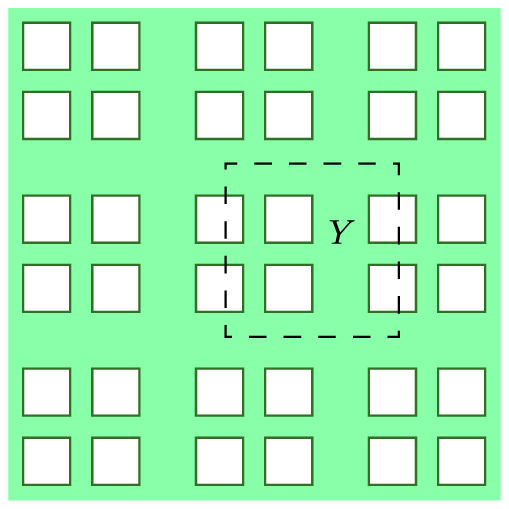}
\end{minipage}
\begin{minipage}[b]{.45\textwidth}
\vspace{.3cm}\centering(a)
\end{minipage}
\begin{minipage}[c]{.05\textwidth}
$\,$
\end{minipage}
\begin{minipage}[b]{.45\textwidth}
\vspace{.3cm}\centering(b)
\end{minipage}
\caption{Two periodic structures (a) and (b) with square symmetric effective elasticity tensors. Periodicity cells are marked with dashed polygons.}
\label{fig:structures}
\end{center}
\end{figure}

%\figgggtex{fig:structures}{
%  Two periodic structures (a) and (b) with square symmetric effective
%  elasticity tensors. Periodicity cells are marked with dashed polygons.
%}{hex}{sqr}

\as{Setting of the problem}

We will consider a periodic structure in the planar linear theory of
  elasticity.
Let $\Omega$ denote a period of the structure.
The associated elasticity tensor is assumed to be homogeneous and isotropic.
The local constitutive relation between the stress $\sigma$ and the strain
  $\e$ will be the standard Hooke law which we write using matrix notation as
  \begin{align}
  \mat{\sii\\\sjj\\\sij} & =
  \mat{K + G & K - G & 0 \\ K - G & K + G & 0 \\ 0 & 0 & G}\!\!\!
  \mat{\eii\\\ejj\\2\eij}.
  \label{eq:hookeslaw}
  \end{align}
The two parameters $K$ and $G$ above are the planar bulk modulus and the
  shear modulus, which here are assumed to be positive real numbers so that the matrix
  is positive definite.
Some additional assumptions will be made on the domain $\Omega$.

For a displacement field $u = (u_1,u_2)$ defined on $\Omega$, with values in
  $\rr^2$, the gradient will be decomposed into its symmetric and
  antisymmetric parts:
  \begin{align*}
  \nabla u & = \e(u) + \omega(u).
  \end{align*}
Explicitly, the components of these are
  \begin{align*}
  \eIJ(u) & = \Inv{2}\lrp{\ddt{u_i}{x_j} + \ddt{u_j}{x_i}}, & 
  \wIJ(u) & = \Inv{2}\lrp{\ddt{u_i}{x_j} - \ddt{u_j}{x_i}}.
  \end{align*}
The stress $\sigma(u)$ is then defined by equation~\re{eq:hookeslaw}.
When the displacement field is clear from the context, it will be dropped
  from the strain, the rotation, and the stress, writing just $\e$, $\omega$,
  and $\sigma$.

The domain $\Omega$ is assumed to be a bounded and connected open
  set in $\rr^2$ with Lipschitz continuous boundary.
Moreover, we assume that there correspond periods $l_i > 0$ such that the
  periodic extension $\Lambda$ of $\Omega$,
  $$
  \Lambda = \mop{Int}\mop{Cl}\bigcup_{k \in \zz^2} ( (l_1k_1, l_2k_2) + \Omega),
  $$
  is also connected and open with Lipschitz continuous boundary.
The translate of the cell $(0,l_1) \times (0,l_2)$ that contains $\Omega$ will be
  denoted by $Y$.

Apart from the above standard hypotheses on the periodicity cell, we will
  restrict our study to the following special type of structure.
We assume that there exist two line segments contained in $\Omega$ that
  connects the opposite faces of $Y$. 
This means that some translates $\gamma_1$ and $\gamma_2$ of
  $(0,l_1) \times \set{\alpha}$ and $\set{\beta} \times (0,l_2)$,
  respectively, are subsets of $\Omega$ for some $\alpha$ and $\beta$.
Illustrations of a perforated cell $\Omega$ inside a cell $Y$ with
  line segments $\gamma_1$ and $\gamma_2$ are shown in
  Figure~\ref{fig:structures}(b) and~\ref{fig:cells}(a).

%\figgggtex{fig:cells}{A cell (a) and its extension together with line segments (b). The
%  periodicity cell $\Omega$ in (a) corresponds to the dashed
%  rectangle $Y$ in Figure~\ref{fig:structures}(b). The line segments marked
%  with arrows in (a) are examples of $\gamma_1$ and $\gamma_2$.
%  In (b) the cell shown in (a) is marked as a colored region and the dashed
%  rectangle shows the translated cell $Y'$. The line segments marked
%  with arrows in (b) are the translates $\gamma_1'$ and $\gamma_2'$ of
%  $\gamma_1$ and $\gamma_2$, respectively, from (a) as defined in the proof
%  of Lemma~\ref{lm:stresslemma}.
%}{cell}{cellp}

\begin{figure}[htbp]
\begin{center}
\begin{minipage}[c]{.45\textwidth}
\centering
\includegraphics[trim = 73mm 120mm 73mm 120mm, clip, width=.9\textwidth]{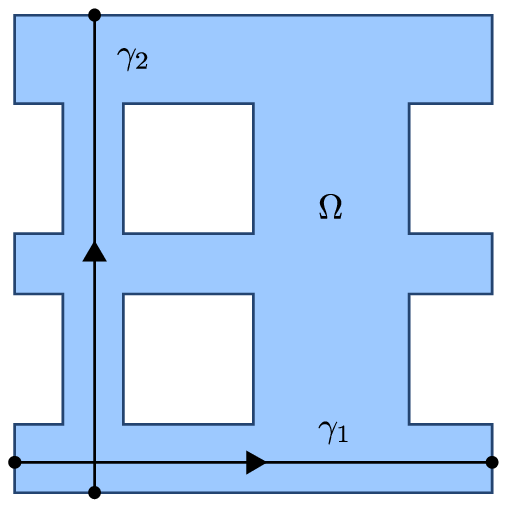}
\end{minipage}
\begin{minipage}[c]{.05\textwidth}
$\,$
\end{minipage}
\begin{minipage}[c]{.45\textwidth}
\centering
\includegraphics[trim = 73mm 120mm 73mm 120mm, clip, width=.9\textwidth]{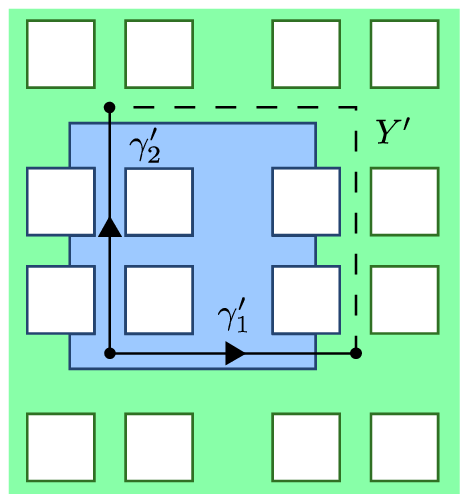}
\end{minipage}
\begin{minipage}[b]{.45\textwidth}
\vspace{.3cm}\centering(a)
\end{minipage}
\begin{minipage}[c]{.05\textwidth}
$\,$
\end{minipage}
\begin{minipage}[b]{.45\textwidth}
\vspace{.3cm}\centering(b)
\end{minipage}
\caption{A cell (a) and its extension together with line segments (b). The
  periodicity cell $\Omega$ in (a) corresponds to the dashed
  rectangle $Y$ in Figure~\ref{fig:structures}(b). The line segments marked
  with arrows in (a) are examples of $\gamma_1$ and $\gamma_2$.
  In (b) the cell shown in (a) is marked as a colored region and the dashed
  rectangle shows the translated cell $Y'$. The line segments marked
  with arrows in (b) are the translates $\gamma_1'$ and $\gamma_2'$ of
  $\gamma_1$ and $\gamma_2$, respectively, from (a) as defined in the proof
  of Lemma~\ref{lm:stresslemma}.}
\label{fig:cells}
\end{center}
\end{figure}

We denote by $H^1_{\mop{per}}(\Omega)$ the closed subspace of $H^1(\Omega)$
  of elements with equal traces on the opposite faces of
  $\partial \Omega \cap \partial Y$.
Let $S$ denote the set of symmetric $2\times 2$ matrices with real entries.
We say that $u \in H^1(\Omega)^2$ is quasiperiodic if
  $u - \xi x \in H^1_{\mop{per}}(\Omega)^2$ for some $\xi \in S$, and $\xi$
  is here called a quasiperiod.

Let $f$ be a function defined on some bounded domain $X$.
The average value of $f$ will be denoted by
  $\di_X f \dx = |X|\inv \int_X f \dx$.  
Any $f$ on $X$ will be considered as a function on $\rr^2$ by periodic
  extension followed by extension by zero.
The relevant example of usage is with $g$ defined on $\Omega$, we will write
  $\di_Y g \dx$ for $|Y|\inv \int_\Omega g \dx$.

The effective elasticity tensor is assumed to be square symmetric and
  positive definite.
This means additional restrictions on $\Omega$ and which can be explicitly stated as
  follows.
The effective tensor, described with parameters $K^*$, $G^*$, and $G^*_{45}$,
  defined for $\xi \in S$ by the equation
  \begin{align}
  \mat{ \di_Y \sii \dx \\ \di_Y \sjj \dx \\ \di_Y \sij \dx } & = 
  \mat{ K^* + G^* & K^* - G^* & 0 \\ K^* - G^* & K^* + G^* & 0 \\ 0 & 0 & G^*_{45} }\!\!\!
  \mat{ \xii \\ \xjj \\ 2\xij },
  \label{eq:effectiverelation}
  \end{align}
  where $\sigma$ is the unique stress obtained from
  equation~\re{eq:hookeslaw} with
  $u - \xi x \in H^1_{\mop{per}}(\Omega)^2$ being a minimizer of the elastic
  energy $\int_\Omega \e(u) \!\cdot\! \sigma(u) \dx$.
The assumption of positive definiteness is equivalent to
  $K^*, \, G^*, \, G^*_{45} > 0$, which are the scaled eigenvalues of the 
  matrix on the right hand side in 
  equation~\re{eq:effectiverelation}.
The matrix is in particular nonsingular.

The existence of a solution $u$ to the above minimization problem is a direct consequence of the Riesz representation theorem and the Korn inequality.
The displacement field $u$ is by periodicity unique up to translation, which shows that the corresponding stress $\sigma$ is unique. 
Moreover, since the data are smooth, $u$ is smooth in some neighborhoods of $\gamma_1$ and $\gamma_2$.
A proof of the regularity of $u$ can be found in~\cite{fichera1972existence} and we refer for contreteness also to~\cite[Chapter~6]{oleinik}.

The relation between the local parameter $K$ and $G$, and the components $c_{ijkl}$ of the elasticity
  compliance\footnote{The inverse of the elasticity tensor.} tensor in the standard notation is given by
  \begin{align*}
  c_{1111} & = c_{2222} = \Inv{4}\kg, &
  c_{1122} & = \Inv{4}\gk, &
  c_{1212} & = \Inv{4G},
  \end{align*}
  where the rest of the components are zero, up to the standard symmetry\footnote{For all indices, $c_{ijkl} = c_{jikl} = c_{ijlk}$.}.
The parameters in relation~\re{eq:effectiverelation} are called the effective planar bulk modulus $K^*$,
  and the effective shear moduli $G^*$ and $G^*_{45}$. 
The relation to the effective compliance tensor is, similarly to the above description, given by
  \begin{align*}
  c^*_{1111} & = c^*_{2222} = \Inv{4}\kgs, &
  c^*_{1122} & = \Inv{4}\gks, &
  c^*_{1212} & = \Inv{4G^*_{45}}.
  \end{align*}
The rest of the components are determined by symmetry as in the local relation.

For a stress field $\sigma$, 
the divergence and the trace will be denoted as follows:
  \begin{align*}
  \mop{div}v & = \lrp{
  \ddt{\sigma_{11}}{x_1} + \ddt{\sigma_{12}}{x_2},
  \ddt{\sigma_{12}}{x_1} + \ddt{\sigma_{22}}{x_2}
  }, &
  \mop{Tr}\sigma & = \sigma_{11} + \sigma_{22}.
  \end{align*}
The outward unit normal of at a Lipschitz continuous boundary will be denoted by $\nu$,
  and the normal traction of the stress by $\sigma\nu$.
The stress and the strain are sometimes represented by the vectors
  $\sigma = (\sii,\sjj,\sij)$ and $\e = (\eii,\ejj,2\eij)$, as in
  equation~\re{eq:hookeslaw} and in a sense also in
  equation~\re{eq:effectiverelation}. The product is then the
  standard inner product on $\rr^3$, which coincides with the Frobenius inner
  product of the tensor fields. The relavant example here is the elastic
  energy above.

\as{Representation of average stress}\label{sec:averagestresses}

In this section we give a representation of the average stress over the
  periodicity cell in terms of the average stress over line segments in the
  definition of the effective elasticity tensor. 
See Figure~\ref{fig:cells}(a) for an illustration of the line segments
  $\gamma_1$ and $\gamma_2$. In Figure~\ref{fig:cells}(b), the other
  line segments used in the proof of the lemma below are illustrated.

\begin{lm}\label{lm:stresslemma}
Let $u$ be a quasiperiodic minimizer of $\int_\Omega \e(u)\!\cdot\!\sigma(u)\dx$.
Then
  \begin{align*}
  \di_Y \sii \dx & = \di_{\gamma_2} \sii \dx,\\ 
  \di_Y \sjj \dx & = \di_{\gamma_1} \sjj \dx,\\
  \di_Y \sij \dx & = \di_{\gamma_2} \sij \dx = \di_{\gamma_1} \sij \dx.
  \end{align*}
  \end{lm}
\begin{proof}
Since $u$ minimizes the elastic energy, the stress $\sigma = \sigma(u)$
  satisfies the equilibrium equation $\mop{div}\sigma = 0$ in the sense of
  distributions with vanishing normal stress $\sigma \nu$ at the boundaries
  of any possible hole in the global structure.
The quasiperiodicity of $u$ implies the periodicity of $\e$ and hence of
  $\sigma$.
We therefore have
  \begin{align*}
  \mop{div}\sigma & = 0 \text{ in $\Lambda$},\\
  \sigma \nu & = 0 \text{ on $\partial \Lambda$}.
  \end{align*}

Let $\gamma_1'$ be the translate of $\gamma_1$ such that the left endpoint of
  $\gamma_1'$ is the common point of $\gamma_1$ and $\gamma_2$.
Let $\gamma_2'$ be the translate of $\gamma_2$ such that its lower endpoint
  is the left endpoint of $\gamma_1'$.
Let $Y'$ be a translate of $Y$ such that $\gamma_1'$ and $\gamma_2'$ are the
  lower and left sides of the boundary of $Y' \cap \Lambda$, respectively.
%%Let $\Omega'$ be the intersection of $\Lambda$ and the interior of
%%  $\gamma_1' \cup \gamma_2' \cup ((l_1,l_2) + \gamma_1' \cup \gamma_2')$.
Let $\Omega' = Y' \cap \Lambda$.

Since $\sigma$ and $\mop{div}\sigma$ have components in $L^2(\Omega')$, 
  the following Green formula holds for all $\varphi \in H^1(\Omega')^2$:
  \begin{align*}
  \int_{\Omega'} \sigma \cdot \nabla \varphi \dx
  +
  \int_{\Omega'} \mop{div}\sigma \cdot \varphi \dx
  =
  \lran{\sigma \nu , \varphi}, 
  \end{align*}
  where $\lran{\cdot,\cdot}$ denotes the pairing of
  $H^{1/2}(\partial \Omega')^2$ and its continuous dual.
Since $\mop{dist}(\partial Y', \partial \Lambda) > 0$ and
  $\sigma\nu = 0$ on $\partial \Lambda$ it follows from the regularity
  of $u$ that $\sigma \nu \in L^2(\partial \Omega')^2$.
  Hence
  $$
  \lran{\sigma \nu, \varphi} = \int_{\partial \Omega'} \sigma \nu \cdot \varphi \dx
  = \int_{\partial Y'} \sigma \nu \cdot \varphi \dx.
  $$

Let $\eta \in \rr$ be such that $\varphi = (x_1 + \eta, 0)$ vanishes on
  $\gamma_2'$.
Then by periodicity and the Green formula,
  \begin{align*}
  \int_{\Omega} \sii \dx & = \int_{\Omega'} \sii \dx = \int_{ \partial Y'} \sigma\nu \cdot (l_1, 0) \dx  
  = l_1 \int_{\gamma_2'} \sii \dx = l_1 \int_{\gamma_2} \sii \dx.
  \end{align*}
A division by $|Y|$ yields the first equation in the statement of the lemma.

The other three equations follow by the same argument when using test functions of
  the types $(0, x_2 + \eta)$, $(x_2 + \eta, 0)$, and $(0, x_1 + \eta)$,
  respectively, for suitable $\eta$. 
\end{proof}

In the proof of the above lemma used the Green formula which
  relies on the existence of a normal trace operator for
  elements in $L^2(\Omega)^n$ with divergence in $L^2(\Omega)$.
We refer the reader to~\cite[Chapter~2]{girault} (and \cite{lionsmagii}) for its construction.

\as{Representation of quasiperiods}\label{sec:quasiperiods}

In this section we give a representation of the quasiperiods of the
  displacement fields in the definition of the effective elasticity tensor
  in terms the average stress over the line segments and the periodicity
  cell.
The idea of the proof is to write $\nabla = \e - \omega$ and then use the
  Newton-Leibniz formula and that the local elasticity tensor is homogeneous
  by supposition.
Hence it is essentially an application of the Ces\`{a}ro formula.

Let $\gamma_1'$ and $\gamma_2'$ denote the translations of $\gamma_1$ and
  $\gamma_2$, respectively, that were considered in the proof of
  Lemma~\ref{lm:stresslemma}. For illustrations see Figure~\ref{fig:cells}(a)
  and~\ref{fig:cells}(b).

\begin{lm}\label{lm:quasiperiodlemma}
Let $u$ be a quasiperiodic minimizer of $\int_\Omega \e(u)\!\cdot\!\sigma(u) \dx$.
Then the quasiperiod $\xi \in S$ of $u$ satisfies
  \begin{align*}
    \xii & = \frac{1}{4K}\lrp{ \di_{\gamma_1} \sii \dx + \di_{Y} \sjj \dx } + \frac{1}{4G}\lrp{ \di_{\gamma_1} \sii \dx - \di_{Y} \sjj \dx },\\
    \xjj & = \frac{1}{4K}\lrp{ \di_{\gamma_2} \sjj \dx + \di_{Y} \sii \dx } + \frac{1}{4G}\lrp{ \di_{\gamma_2} \sjj \dx - \di_{Y} \sii \dx },\\
    \xij & =  \frac{1}{2G}\di_Y \sij \dx + \frac{1}{8}\kg\lrp{\di_{\gamma_1'} \ddt{\mop{Tr}\sigma}{x_2} x_1 \dx 
    + \di_{\gamma_2'} \ddt{\mop{Tr}\sigma}{x_1} x_2 \dx}.
  \end{align*} 
\end{lm}
\begin{proof}
Since $u$ is quasiperiodic we have
by the Newton-Leibniz formula and the local elastic relation~\re{eq:hookeslaw},
\begin{align*}
\xii l_1 & = u_1(b) - u_1(a) = \int_{\gamma_1} \ddt{u_1}{x_1} \dx = 
\int_{\gamma_1} \eii \dx \\
 & =
  \Inv{4K}\lrp{ \int_{\gamma_1} \sii \dx + \int_{\gamma_1} \sjj \dx }
+  \Inv{4G}\lrp{ \int_{\gamma_1} \sii \dx - \int_{\gamma_1} \sjj \dx },
\end{align*}
where $a$ and $b$ denote the endpoints of $\gamma_1$.
By Lemma~\ref{lm:stresslemma},
\begin{align*}
\xii l_1 & = 
   \Inv{4K}\lrp{ \int_{\gamma_1} \sii \dx + \Inv{l_2}\int_{Y} \sjj \dx }
+  \Inv{4G}\lrp{ \int_{\gamma_1} \sii \dx - \Inv{l_2}\int_{Y} \sjj \dx }.
\end{align*}
A division by $l_1$ gives the first representation and the second follows by
  an interchange of the indices.

Denote by $a$ the common point of\footnote{The point $a$ is also the common point
  of $\gamma_1'$ and $\gamma_2'$.} $\gamma_1$ and $\gamma_2$.
With $b$ being the other endpoint of $\gamma_1'$, we have by periodicity,
\begin{align*}
\xij l_1 & = u_2(b) - u_2(a) = \int_{\gamma_1'} \ddt{u_2}{x_1} \dx
= \int_{\gamma_1'} \eij \dx  + \int_{\gamma_1'} \wji \dx \\
& = \Inv{2G}\int_{\gamma_1'} \sij \dx  + \int_{\gamma_1'} \wji \dx
 = \Inv{2G}\int_{\gamma_1} \sij \dx  + \int_{\gamma_1'} \wji \dx.
\end{align*}
An integration by parts gives by periodicity,
\begin{align*}
\int_{\gamma_1'} \wji \dx & =
l_1 \wji(a) - \int_{\gamma_1'} \ddt{\wji}{x_1}(x_1 - b_1) \dx = 
l_1 \wji(a) - \int_{\gamma_1'} \ddt{\wji}{x_1}x_1 \dx.
\end{align*}
By the local relation~\re{eq:hookeslaw},
\begin{align*}
\int_{\gamma_1'} \ddt{\wji}{x_1}x_1 \dx & =
\int_{\gamma_1'} \lrpb{ \ddt{\eij}{x_1} - \ddt{\eii}{x_2} } x_1 \dx \\
& = \Inv{2G}\int_{\gamma_1'} \ddt{\sij}{x_1} x_1 \dx
  - \Inv{4}\kg \int_{\gamma_1'} \ddt{\sii}{x_1} x_2 \dx \\
& \quad  - \Inv{4}\gk \int_{\gamma_1'} \ddt{\sjj}{x_2} x_1 \dx.
\end{align*}
A division by $l_1$ yields
  \begin{align*}
  \xij & = \wji(a)
  + \Inv{2G}\dint_{\gamma_1} \sij \dx 
  - \Inv{2G}\dint_{\gamma_1'} \ddt{\sij}{x_1} x_1 \dx \\
  & \quad + \Inv{4}\kg \dint_{\gamma_1'} \ddt{\sii}{x_2} x_1 \dx 
  + \Inv{4}\gk \dint_{\gamma_1'} \ddt{\sjj}{x_2} x_1 \dx \\
  & = \wji(a) + \Inv{2G}\di_Y \sij \dx + \Inv{4} \kg \dint_{\gamma_1'} \ddt{\mop{Tr}\sigma}{x_2} x_1 \dx,
  \end{align*}
  where we in the last step used Lemma~\ref{lm:stresslemma} and that
  $\mop{div}\sigma = 0$.
By periodicity $\wij(a) = \wij(b)$ and thus we get by an interchange of the indices,
  \begin{align*}
  \xij & = \wij(a) + \Inv{2G}\di_Y \sij \dx + \Inv{4} \kg \dint_{\gamma_2'} \ddt{\mop{Tr}\sigma}{x_1} x_2 \dx.
  \end{align*}
By the antisymmetry of $\omega$, taking the arithmetic mean of these two
  representations of $\xij$ completes the proof.
\end{proof}

\as{The Vigdergauz equations}\label{sec:vigdergauz}

In this section we give a proof of the Vigdergauz equations for the considered
  class of structures.
The proof is a straightforward application of the representation lemmas of
  the previous sections to the definition of the effective elasticity tensor.
The independence of the local elastic properties $K$ and $G$ follow from
  the Michell lemma.
We briefly recall the requirements on the periodicity cell.
The cell $\Omega$ is assumed to be such that the effective elasticity tensor
  is positive definite and has square symmetry so that it can be
  described by three parameters $K^*$, $G^*$, and $G^*_{45}$.
Moreover, it is assumed that there exist line segments in $\Omega$,
  as in Figure~\ref{fig:cells}(a),
  connecting the opposite sides without touching the boundary of any possible hole
  in the periodic extension $\Lambda$ of $\Omega$.

The Michell lemma states that for the pure traction problem\footnote{The
  Neumann problem in which only the normal traction $\sigma \nu$ is specified
  on the boundary and no external body forces are present.
In addition, the total force on each hole is required to vanish.} in $\rr^2$ with
  homogeneous and isotropic elasticity tensor on a bounded Lipschitz domain,
  the stress field is independent of the values of the components of the
  elasticity tensor.
A reference is given in the end of this section.

\begin{tm}\label{tm:vigdergauz}
There exist nonnegative real numbers $A_i$ depending only on $\Omega$ such
  that
  \begin{align*}
  \Inv{K^*} & = \Inv{K} + A_1 \! \kg,\\
  \Inv{G^*} & = \Inv{G} + A_2 \! \kg,\\
  \Inv{G^*_{45}} & = \Inv{G} + A_3 \! \kg.
  \end{align*}
\end{tm}
\begin{proof}
Let $K, G > 0$ be given.
Let $\eta \in S$.
Then by the nonsingularity of the effective tensor,
  there exists a quasiperiodic minimizer of
  $\int_\Omega \e(u)\!\cdot\!\sigma(u)\dx$
  with unique quasiperiod $\xi \in S$ and 
  with average stress, defined by the left hand side of
  equation~\re{eq:effectiverelation}, which is equal to $\eta$
  represented by $(\eta_{11},\eta_{22},\eta_{12})$.
  The displacement field $u$ is unique up to translation.
  Therefore the stress $\sigma$ is unique.
  Thus $\sigma$ is the unique solution to the corresponding
  pure traction problem.
  Hence by the Michell lemma, the stress $\sigma$ is independent 
  of both $K$ and $G$.
We proceed in three steps.

Choose $\sigma$, by specifying $\eta$, such that
  \begin{align*}
  \dint_Y \sii \dx & \neq 0, &
  \dint_Y \sjj \dx & = \dint_Y \sij \dx = 0.
  \end{align*}
By equation~\re{eq:effectiverelation} and Lemma~\ref{lm:quasiperiodlemma},
  \begin{align}
  \Inv{K^*} + \Inv{G^*} = \frac{4\xii}{\dint_Y \sii \dx} = \frac{\dint_{\gamma_1}\sii \dx}{\dint_Y \sii \dx}\kg.
  \label{eq:aux1}
  \end{align}

Choose $\sigma$ such that
  \begin{align*}
  \dint_Y \sii \dx & = \dint_Y \sjj \dx \neq 0, &
  \dint_Y \sij \dx & = 0.
  \end{align*}
Then by equation~\re{eq:effectiverelation} and Lemma~\ref{lm:quasiperiodlemma},
  \begin{align}
  \Inv{K^*} = \frac{2\xii}{\dint_Y \sii \dx} = \Inv{K} + \Inv{2} \! \lrp{ \frac{\dint_{\gamma_1} \sii \dx}{\dint_Y \sii \dx} - 1 } \! \kg.
  \label{eq:aux2}
  \end{align}
From the equations~\re{eq:aux1} and~\re{eq:aux2}, the two first equations in
  the statement of the lemma follow, since the stress is independent of the local
  elastic properties.

Choose $\sigma$ such that
  \begin{align*}
  \dint_Y \sii \dx & = \dint_Y \sjj \dx = 0, &
  \dint_Y \sij \dx & \neq 0.
  \end{align*}
From equation~\re{eq:effectiverelation} and Lemma~\ref{lm:quasiperiodlemma}
  it follows that
  \begin{align}
  \Inv{G^*_{45}} = \Inv{G} + 
  \frac{ \dint_{\gamma_1'} \ddt{\mop{Tr}\sigma}{x_2} x_1 \dx + \dint_{\gamma_2'} \ddt{\mop{Tr}\sigma}{x_1} x_2 \dx }{\dint_Y \sij \dx}\kg,
  \label{eq:aux3}
  \end{align}
  which completes the proof of the separation of variables.

Left to show is the nonnegativity of the geometric constants.
This follows directly from the assumption of positive definite effective elasticity
tensor. For the matrix is symmetric and thus has strictly positive real eigenvalues,
so the same holds for its inverse.
Using the equations \re{eq:aux1}--\re{eq:aux3} to write the components of the effective compliance matrix in terms of $A_1, \, A_2, \, A_3$, and $K, \, G$, we find that
the eigenvalues $\lambda_i$ are 
\begin{align*}\label{}
  \lambda_1 & = \frac{(K+G)A_1 + G}{2KG}, &
  \lambda_2 & = \frac{(K+G)A_2 + K}{2KG}, &
  \lambda_3 & = \frac{(K+G)A_3 + K}{KG},
\end{align*}
for any choice of $K,\,G > 0$. 
Hence $A_1, \, A_2, \, A_3 \ge 0$.
\end{proof}

\begin{rem}
There are natural alternatives to letting $\sigma$ be such that
  its average value over the periodicity cell is some scaled element of
  the canonical basis of $\rr^3$ as in the proof of Theorem~\ref{tm:vigdergauz}.
For example, one can choose $\sigma$ such that
  the average stress vector varies over eigenvectors of the effective
  elasticity matrix in equation~\re{eq:effectiverelation},
  which are clearly also the eigenvectors of the compliance matrix.
In this way the deformations are connected to the physical intuition 
  of what loads the bulk and shear moduli are supporting
  and in that sense arrive more directly to the Vigdergauz equations as
  examplified in~\cite{vigder}. 
However, the proof given above would not be simplified in any way by making such
  modification.
\end{rem}

\begin{rem}\label{}
The nonnegativity of the geometric constants $A_1, \, A_2, \, A_3$
in Theorem~\ref{tm:vigdergauz} can also be seen as a direct consequence of the
Hashin-Shtrikman bounds, which for 
  square symmetric tensors can be written
  \begin{align*}
  \Inv{K^*} & \ge \Inv{K} + \frac{1 - \rho}{\rho}\kg,\\
  \Inv{G^*} & \ge \Inv{G} + \frac{1 - \rho}{\rho}\kg,\\
  \Inv{G^*_{45}} & \ge \Inv{G} + \frac{1 - \rho}{\rho}\kg,
  \end{align*}
  where $\rho$ denotes the positive volume fraction $\rho = |\Omega|/|Y| \le 1$.
\end{rem}

A proof of the Michell lemma for smooth domains can be found in~\cite[Chapter~5]{musk}.
For Lipschitz domains, the theorem of Stokes can be used.
We refer the reader to~\cite[Chapter~2]{girault} for an appropriate
  version of the theorem, which should be applied to the simply connected
  components of a decomposition of $\Omega$.
The Hashin-Shtrikman bounds in terms of $K^*$, $G^*$, and $G^*_{45}$,
  that were used in the above proof, can be found in~\cite{bergren2}.

Remark that if some lines at angles $\pi/4$ and $3\pi/4$ transverse the
  corresponding periodicity cell without touching the boundary of any possible hole,
  the proof for $G^*_{45}$ simplifies to that of $G^*$.
This is because by the hypothesis of the symmetry of the effective tensor,
  $G^*$ and $G^*_{45}$ interchange roles
  upon a rotation of the coordinate system by $\pi/4$.
The structure and cell illustrated in Figure~\ref{fig:structures}(b) constitute an
  example of a structure for which no such lines exist.

\bibliographystyle{plain}
\bibliography{refs}

\def\cprime{$'$} \def\cprime{$'$} \def\cprime{$'$}
\begin{thebibliography}{1}

\bibitem{bergren2}
S.~A. Berggren, D.~Lukkassen, A.~Meidell, and L.~Simula.
\newblock Some methods for calculating stiffness properties of periodic
  structures.
\newblock {\em Appl. Math.}, 48(2):97--110, 2003.

\bibitem{fichera1972existence}
G.~Fichera.
\newblock {Existence theorems in elasticity}.
\newblock {\em Handbuch der Physik}, VIa/2:347--424, 1972.

\bibitem{girault}
V.~Girault and P.-A. Raviart.
\newblock {\em Finite element methods for {N}avier-{S}tokes equations},
  volume~5 of {\em Springer Series in Computational Mathematics}.
\newblock Springer-Verlag, Berlin, 1986.
\newblock Theory and algorithms.

\bibitem{lionsmagii}
J.-L. Lions and E.~Magenes.
\newblock {\em Non-homogeneous boundary value problems and applications. {V}ol.
  {II}}.
\newblock Springer-Verlag, New York, 1972.
\newblock Translated from the French by P. Kenneth, Die Grundlehren der
  mathematischen Wissenschaften, Band 182.

\bibitem{musk}
N.~I. Muskhelishvili.
\newblock {\em Some basic problems of the mathematical theory of elasticity.
  {F}undamental equations, plane theory of elasticity, torsion and bending}.
\newblock P. Noordhoff Ltd., Groningen, 1953.
\newblock Translated by J. R. M. Radok.

\bibitem{oleinik}
O.~A. Ole{\u\i}nik, A.~S. Shamaev, and G.~A. Yosifian.
\newblock {\em Mathematical problems in elasticity and homogenization},
  volume~26 of {\em Studies in Mathematics and its Applications}.
\newblock North-Holland Publishing Co., Amsterdam, 1992.

\bibitem{vigder}
S.~Vigdergauz.
\newblock Complete elasticity solution to the stress problem in a planar
  grained structure.
\newblock {\em Math. Mech. Solids}, 4(4):407--439, 1999.

\end{thebibliography}
\end{document}